\documentclass[a4paper,12pt]{article}
\usepackage{amsfonts}
\usepackage{amssymb}
\usepackage{latexsym}
\usepackage{amsmath}
\usepackage{amsthm}
\usepackage{graphicx}
\usepackage[usenames]{xcolor}

\definecolor{darkred}{RGB}{139,0,0}
\definecolor{darkgreen}{RGB}{0,100,0}
\definecolor{darkmagenta}{RGB}{139,0,139}
\definecolor{darkorange}{RGB}{190,70,20}

\newtheorem{theorem}{Theorem}
\newtheorem{lemma}[theorem]{Lemma}
\newtheorem{corollary}[theorem]{Corollary}
\newtheorem{proposition}[theorem]{Proposition}

\theoremstyle{definition}

\newtheorem{remark}{Remark}

\newcommand{\NN}{\mathbb{N}}
\newcommand{\RR}{\mathbb{R}}

\newcommand{\bsalpha}{\boldsymbol{\alpha}}

\newcommand{\bsx}{\boldsymbol{x}}
\newcommand{\bsy}{\boldsymbol{y}}
\newcommand{\bsz}{\boldsymbol{z}}

\newcommand{\bsb}{\boldsymbol{b}}

\newcommand{\cP}{\mathcal{P}}

\newcommand{\vol}{{\rm Vol}}
\newcommand{\diam}{\mathrm{diam}}
\newcommand{\dist}{\mathrm{dist}}

\title{On the relation of the spectral test to isotropic discrepancy and $L_q$-approximation in Sobolev spaces}
\author{Mathias Sonnleitner and Friedrich Pillichshammer\thanks{The authors are supported by the Austrian Science Fund (FWF): Project F5509-N26 (Pillichshammer) and projects F5513-N26 and P32405 (Sonnleitner)
}}

\date{}

\begin{document}

\maketitle

\begin{abstract}
This paper is a follow-up to the recent paper  ``A note on isotropic discrepancy and spectral test of lattice point sets'' [J. Complexity, 58:101441, 2020].  We show that the isotropic discrepancy of a lattice point set is at most $d \, 2^{2(d+1)}$ times its spectral test, thereby correcting the dependence on the dimension $d$ and an inaccuracy in the proof of the upper bound in Theorem~2 of the mentioned paper. The major task is to bound the volume of the neighbourhood of the boundary of a convex set contained in the unit cube. 

Further, we characterize averages of the distance to a lattice point set in terms of the spectral test. As an application, we infer that the spectral test -- and with it the isotropic discrepancy -- is crucial for the suitability of the lattice point set for the approximation of Sobolev functions.
\end{abstract}

\centerline{\begin{minipage}[hc]{130mm}{
}
\end{minipage}}

\section{Introduction and statement of the results}

In our recent paper \cite{PS20} we exhibited a close connection between the isotropic discrepancy and the spectral test of lattice point sets in the $d$-dimensional unit cube $[0,1)^d$. For the definition of these well-established notions we refer to Hellekalek \cite{H98} as well as \cite{PS20}. The central result is \cite[Theorem~2]{PS20} which states that the isotropic discrepancy $J_N$ of an $N$-element lattice point set $\cP(L)$ in $[0,1)^d$ and the spectral test $\sigma(L)$ are -- up to multiplicative factors only depending on the dimension $d$ -- equivalent, i.e., we have $J_N(\cP(L)) \asymp_d \sigma(L)$. Possible choices for the involved multiplicative factors are stated explicitly. (The asymptotic notation $A(N)\asymp B(N)$ means that there exist numbers $0<c<C$ such that $c A(N)\le B(N)\le C A(N)$ for all $N\in \mathbb{N}$. Furthermore, $\asymp_d$ indicates that the numbers $c$ and $C$ may only depend on $d$.) 

Unfortunately, it just turned out that the proof of the corresponding upper bound is flawed since the employed argument of Aistleitner et al., used in the proof of \cite[Lemma~17]{ABD12}, to estimate the discrepancy of a single set in terms of the number of cells intersecting its boundary, was incorrectly extended to higher dimensions in \cite[page~5, lines~19-21]{PS20},
although the asymptotic result itself remains valid. 

That kind of argument, which boils down to bounding the measure of the neighbourhood of certain sets, is very useful in discrepancy theory and appears often in the literature, even though sometimes superficially (especially in high dimensions). This is true in particular in the context of jittered or stratified sampling which can be used to derive upper bounds on the discrepancy with respect to various set systems, see, e.g., the book of Beck and Chen \cite[Chapter~8]{BC08}. There, in equation (3), an estimate on the number of intersecting cells was given, however with a hidden constant. Confer also Drmota and Tichy \cite[Section~2.1.2]{DT97} for another presentation. In \cite[Section~8]{BCC+19}, Brandolini et al.\;extend this method to metric measure spaces, a key hypothesis being a bound on the measure of the neighbourhood of the boundary of the involved sets. Similar to Corollary~8.3 there, we shall derive such a bound for convex sets and give an explicit constant (see Corollary~\ref{co1} in Section~\ref{sec3}) with the difference being that we do not intersect convex sets with $[0,1]^d$.

Using a slightly different setup than \cite{PS20}, Aistleitner et al.\;considered translates of fundamental cells which they intersected with the unit square. However, also their approach requires some modification since in \cite{ABD12} the last equality on page 1006 and the first on page 1007 are not correct. This can be repaired if one replaces $[0,1]^2$ by $\RR^2$ in the definiton of the unit cells. To obtain the necessary volume estimates, one could look into, e.g., the proof of Lemma~\ref{le3} below to see that \cite[Lemma~17]{ABD12} holds if $4$ is replaced by $4+\sqrt{8}\pi$ (which can further be replaced by 5, as the isotropic discrepancy is in any case at most 1).\\

The aim of this note is twofold: First, we aim at working out the details for a correct adjustment of the used argument to higher dimensions, thereby correcting the dimension dependence of the upper bound in \cite[Theorem~2]{PS20} as follows.

\begin{theorem}\label{thm1}
Let $\mathcal{P}(L)$ be an $N$-element lattice point set in $[0,1)^d$. Then we have 
\[
J_N(\mathcal{P}(L))\le d \,2^{2(d+1)}\,\sigma(L)\qquad (\text{instead of }d^2\,2^d\,\sigma(L)).
\]
\end{theorem}

The second aim is to show a close relation between the spectral test of a lattice point set and the worst-case error of $L_q$-approximation for functions from Sobolev spaces. For this, we need to introduce some relevant concepts.

Let $p\in [1,\infty]$ and $s>d/p$ be an integer. Then the Sobolev space $W^s_p(\RR^d)$ is the collection of all $L_p$-functions for which the expression 
\[
\|f\|_{W^s_p(\RR^d)}:=\Big(\sum_{|\bsalpha|\le s} \|D^{\bsalpha}f\|_{L_p(\RR^d)}^p\Big)^{1/p}
\]
is finite, where the summation is extended over all multi-indices $\bsalpha=(\alpha_1,\ldots,\alpha_d) \in \mathbb{N}_0^d$, for which $|\bsalpha|=\alpha_1+\cdots+\alpha_d$ is less or equal the so-called smoothness parameter $s$. The notation $D^{\bsalpha} f$ indicates the weak mixed partial derivatives of $f$ of order $\alpha_i$ in coordinate $x_i$ for $i \in \{1,\ldots,d\}$, i.e., $$D^{\bsalpha} f := \frac{\partial^{|\bsalpha|}f}{\partial^{\alpha_1} x_1 \ldots \partial^{\alpha_d} x_d}.$$ 

For every function in $W^s_p(\RR^d)$ we find a continuous representative since $s>d/p$. On the cube $[0,1)^d$ we define $W^s_p([0,1)^d)$ by collecting all functions $f:[0,1)^d\to\RR$ for which $f(x)=g(x)$ for all $x\in [0,1)^d$ and some $g\in W^s_p(\RR^d)$. Taking the infimum over all such $g$, we define $$\|f\|_{W^s_p([0,1)^d)}:=\inf \|g\|_{W^s_p(\RR^d)}.$$

Let $\cP=\{\bsx_0,\ldots,\bsx_{N-1}\}\subseteq [0,1)^d$ be a finite point set, which may be an $N$-element lattice point set. Let $q\in [1,\infty]$. We consider algorithms $A_{\cP,\varphi}$ taking $f\in W^s_p([0,1)^d)$ to
\begin{equation}\label{eq:algo}
A_{\cP,\varphi}(f)=\varphi\bigl(f(\bsx_0),\ldots,f(\bsx_{N-1})\bigr),\quad \text{where } \varphi:\mathbb{R}^N\to L_q([0,1)^d),
\end{equation}
which approximate the continuous embedding $W^s_p([0,1)^d)\hookrightarrow L_q([0,1)^d)$ well in the sense of having a small worst-case error
\[
e(A_{\cP,\varphi},d,s,p,q):=\sup_{\|f\|_{W^s_p([0,1)^d)}\le 1} \|f-A_{\cP,\varphi}(f)\|_{L_q([0,1)^d)}.
\]
The best we can do with the given point set $\cP$ is measured by
\[ 
e(\cP,d,s,p,q):=\inf_{\varphi} e(A_{\cP,\varphi},d,s,p,q),
\]
which is the minimal error over all algorithms of the form \eqref{eq:algo}. It is known, see, e.g., Novak and Triebel \cite{NT06} or Novak and Wo\'zniakowski \cite[Remark 4.43]{NW08}, that the minimal error we can achieve using any $N$ points satisfies
\[
e(N,d,s,p,q):=\inf_{|\cP|= N} e(\cP,d,s,p,q)\asymp N^{-s/d+(1/p-1/q)_+}
\] 
where the infimum is over all $N$-point sets in $[0,1)^d$ and $(a)_+=\max\{a,0\}$ for $a\in\mathbb{R}$.  For a subsequence $(N_k)_{k\in\NN}$ of $\NN$ a sequence of point sets $(\cP_{N_k})_{k\in\mathbb{N}}$ in $[0,1)^d$, where $|\cP_{N_k}|=N_k$,  behaves asymptotically optimal if $e(\cP_{N_k},d,s,p,q)$ decays at this rate.

Rescalings of the integer lattice perform optimally for the approximation of Sobolev functions; a proof may be deduced from, e.g., \cite{NT06}. In the recent paper \cite{KS20}  by Krieg and the first author, a characterization was proven, showing that for general sequences of point sets optimality is determined by an $L_{\gamma}$-norm of the distance function $\dist(\cdot,\cP)$ for some $\gamma \in (0,\infty]$ which depends on $s,p$ and $q$.
 For more information on approximating Sobolev functions using sample values we refer to the references given in \cite{KS20}. The following theorem provides an asymptotic characterization of the minimal error achievable using lattice point sets in terms of their spectral test. 

\begin{theorem}\label{thm2}
Let $d \in \NN$, $p,q\in [1,\infty]$ and $s>d/p$. For every lattice point set $\mathcal{P}(L)$ in $[0,1)^d$ it holds that
\[
 e(\mathcal{P}(L),d,s,p,q)\,
 \asymp_{d,s,p,q} \,\sigma(L)^{s-d(1/p-1/q)_+}.
\]
Here, the asymptotic notation conceals constants indepedent of the integration lattice $L$.
\end{theorem}

The proof of Theorem~\ref{thm2} will be provided in Section~\ref{sec4}.\\

Given a sequence of integration lattices $(L_{N_k})_{k\in\NN}$ in $\mathbb{R}^d$, the spectral test $\sigma(L_{N_k})$ cannot decay faster than $N_k^{-1/d}$, where $N_k=|\cP(L_{N_k})|$ and there exist sequences achieving this rate, see, e.g., \cite[Propositions~3~and~4]{PS20}. Thus, we can interpret Theorem~\ref{thm2} in the following way: 
\begin{corollary}\label{cor:optimal-lattices}
Lattice point sets are asymptotically optimal for the $L_q$-approximation of Sobolev functions from $W^s_p([0,1)^d)$ if and only if the spectral test behaves optimally, i.e, 
\[
e(\cP(L_{N_k}),d,s,p,q)\asymp e(N_k,d,s,p,q) \qquad \emph{if and only if} \qquad \sigma(L_{N_k})\asymp N_k^{-1/d}. 
\] 
\end{corollary}

Further, by \cite[Theorem~2]{PS20} one can replace ``spectral test'' by ``isotropic discrepancy'' in Theorem \ref{thm2} and Corollary \ref{cor:optimal-lattices}.

\section{The volume of parallel bodies and the proof of Theorem~\ref{thm1}}\label{sec3}

We follow \cite[Proof of Theorem~2]{PS20} but we only highlight the necessary additional arguments.  By definition, the normalized isotropic discrepancy satisfies $J_N(\mathcal{P}(L))\le 1$, and thus we can assume that $$\sigma(L)\le \frac{1}{d \, 2^{2(d+1)}}.$$ By the second display on \cite[page~5]{PS20} we have 
\[
\mathrm{diam}(P)\le d\, 2^{d-1}\sigma(L),
\]
where $P$ is the fundamental parallelotope with respect to a reduced basis of $L$ as given by the LLL-algorithm and $\mathrm{diam}$ denotes the diameter of a set measured in Euclidean norm.
Therefore, without loss of generality,
\[
\mathrm{diam}(P)\le d\,2^{d-1}\frac{1}{d\, 2^{2(d+1)}}=\frac{1}{2^{d+3}}< 1.
\]

Let $K\subseteq [0,1]^d$ be non-empty and convex and $K^C=\mathbb{R}^d\backslash K$ its complement. As our interest lies in the volume of the following sets, we may assume that $K$ is closed, i.e., $K$ is a convex body in the sense of Schneider \cite{Sch14}. We define for every $\rho\ge 0$ the sets
\[
K_{\rho}^+:=\{\bsx\in K^C \ : \ \dist(\bsx,K)\le \rho\}\quad\text{and}\quad K_{\rho}^-:=\{\bsx\in K \ : \ \dist(\bsx,K^C)\le \rho\},
\] 
which split the neighbourhood $\{\bsx\in\mathbb{R}^d \ : \ \dist(\bsx,\partial K)\le \rho\}$ into a part outside of $K$ and a part inside of $K$. Here, $\dist(\bsx,A):=\inf_{\bsy\in A} \|\bsx-\bsy\|_2$ for $\bsx\in\mathbb{R}^d$ and $A\subseteq \mathbb{R}^d$.

Let $\vol$ denote the $d$-dimensional volume. Then it remains to establish the bound $\max\{\vol(K_{\rho}^+),\vol(K_{\rho}^-)\}\le 2^{d+3}\rho$ for $\rho\le 1$, in which we then set $\rho=\diam(P)$.

For this, we shall employ well-known arguments from convex geometry, e.g., taken from the book \cite{Sch14}. Given non-empty $A,B\subseteq \mathbb{R}^d$ we define the Minkowski addition and the Minkowski difference by
\[
A+B:=\bigcup_{\bsb\in B} (A+\bsb)\quad \text{and}\quad A\div B:=\bigcap_{\bsb\in B} (A-\bsb),\text{ respectively.}
\]
Let $B$ be the (open) unit ball of $(\mathbb{R}^d,\|\cdot\|_2)$. Then, for all $\rho> 0$,
\begin{align*}
K+\rho B&=\{\bsx\in\mathbb{R}^d \ : \ \dist(\bsx,K)<\rho\},\\
K\div\rho B&=\{\bsx\in K \ : \ \dist\big(\bsx,K^C\big)\ge \rho\}.
\end{align*}
We define a family of convex parallel sets by
\[
K_{\rho}:=
\begin{cases}
K+\rho B& \text{for }\rho\ge 0,\\
K\div (-\rho)B & \text{for } \rho < 0.
\end{cases}
\]
The largest $\rho>0$ such that $K\div \rho B\neq\emptyset$ is given by the inradius of $K$, which is defined by $r(K):=\sup\{\rho\ge 0 \ : \ \bsx+\rho B\subseteq K \text{ for some } \bsx\in \mathbb{R}^d\}$. As a consequence, $\vol(K_{-r(K)})=0$ and $\rho<-r(K)$ implies $K_{\rho}=\emptyset$. For $\rho> 0$ we have $K_{\rho}=K+\rho B$ and $K_{-\rho}=K\div \rho B$. Further, $K_0=K$. 

Comparing the definitions we see that for any $\rho\ge 0$
\begin{equation}\label{eq:volumedifference}
\vol(K_{\rho}^+)=\vol(K_{\rho})-\vol(K) \quad\text{and}\quad \vol(K_{\rho}^-)=\vol(K)- \vol(K_{-\rho}).
\end{equation}

We will use Steiner's formula (see, e.g., \cite[Eq.~(4.8)]{Sch14}) stating that for every $\rho\ge 0$

\begin{equation}\label{eq:steiner}
\vol(K+\rho B)=\sum_{j=0}^d \binom{d}{j}W_j(K)\rho^j,
\end{equation}
where $W_j(K)$ is the $j$-th quermassintegral of $K$. As a mixed volume, it is monotonous under set inclusion, i.e., it satisfies $W_j(K_1)\le W_j(K_2)$ for $j=0,\ldots, d$, whenever $K_1\subseteq K_2$ are convex bodies. Note that $W_0(K)=\vol(K)$ and $d\, W_1(K)$ is the surface area of $K$.

We also need a result noted by Hadwiger in his book \cite[Eq.~(30), page 207]{Had57}; compare also to \cite[Proposition~2.6]{RG20} by Richter and G\'omez who give additional references. 

\begin{lemma}\label{lem:diffvol}
The function $v(\rho):=\vol(K_{\rho})$ is differentiable on $(-r(K),\infty)$ and its derivative satisfies $v'(\rho)=d\, W_1(K_{\rho})$.
\end{lemma}

From this we derive the following inequality.

\begin{lemma}\label{lem:outervsinner}
For all $\rho\ge 0$ we have $\vol(K_{\rho}^+)\ge \vol(K_{\rho}^-).$
\end{lemma}
\begin{proof}
Using \eqref{eq:volumedifference}, this inequality can be written as $v(\rho)-v(0)\ge v(0)-v(-\rho)$. Suppose first that $0<\rho\le r(K)$. Lemma \ref{lem:diffvol} and the mean value theorem yield some $\rho_1\in (0,\rho)$ and $\rho_2\in (-\rho,0)$ such that
\[
\frac{v(\rho)-v(0)}{\rho}=v'(\rho_1) \quad \text{and}\quad \frac{v(0)-v(-\rho)}{\rho}=v'(\rho_2).
\]
Since $K_{\rho_2}\subseteq K_{\rho_1}$ and the quermassintegral $W_1(\cdot)$ is monotonous, we have $v'(\rho_1)\ge v'(\rho_2)$. This completes the proof in this case.

If $\rho=0$, we have equality by definition, and if $\rho>r(K)$, the monotonicity of the volume yields $v(r(K))\le v(\rho)$, and thus from the previously established case $\rho=r(K)$ it follows that
\[
v(\rho)-v(0)\ge v(r(K))-v(0)\ge v(0)-v(-r(K))=v(0)-v(-\rho)
\]
since $v(-r(K))=v(-\rho)=0$. By means of 
 \eqref{eq:volumedifference} this completes the proof.
\end{proof}

We finalize our discussion by establishing:
\begin{lemma}\label{le3}
For any $\rho\in [0,1]$ we have $\max\{\vol(K_{\rho}^+),\vol(K_{\rho}^-)\}\le 2^{d+3}\rho$.
\end{lemma}
\begin{proof}
Lemma \ref{lem:outervsinner} implies that 
 $\max\{\vol(K_{\rho}^+),\vol(K_{\rho}^-)\}= \vol(K_{\rho}^+)$ and it remains to estimate the latter. By Steiner's formula \eqref{eq:steiner} we have
\[
\vol(K_{\rho}^+)=\vol(K_{\rho})-\vol(K)=\sum_{j=1}^d \binom{d}{j} W_j(K) \rho^j.
\]
The monotonicity of the quermassintegrals yields
\[
\vol(K_{\rho}^+)\le \sum_{j=1}^d \binom{d}{j} W_j([0,1]^d) \rho^j
\]
with equality for $K=[0,1]^d$. According to Lotz et al.\ \cite[Example~1.3]{LMN+19} it is a classical fact that for $j=0,1,\ldots, d$ the $j$-th intrinsic volume $V_j([0,1]^d)$ of the unit cube equals $\binom{d}{j}$. The relation to the quermassintegrals is given by $\binom{d}{j}W_j([0,1]^d)=\kappa_jV_{d-j}([0,1]^d),$ where $\kappa_j$ is the $j$-dimensional volume of the unit ball of $(\mathbb{R}^j,\|\cdot\|_2)$. Together with the symmetry of the binomial coefficients, this implies $W_j([0,1]^d)=\kappa_j$. Therefore, as $\rho\le 1$, we have
\[
\vol(K_{\rho}^+)\le \rho\sum_{j=1}^d \binom{d}{j}\kappa_j \rho^{j-1} \le \rho\sum_{j=1}^d\binom{d}{j}\kappa_j.
\]
We complete the proof by using the fact that $\kappa_j\le \kappa_5=8\pi^2/15\le 2^3$ for every $ j\in \NN$ and that $\sum_{j=0}^d \binom{d}{j}=2^d$.
\end{proof}

The estimate from Lemma~\ref{le3} with $\rho={\rm diam}(P) \le d\, 2^{d-1} \sigma(L)$ has to be employed in \cite[page~5, lines 19-21]{PS20} which yields Theorem~\ref{thm1} and corrects  \cite[Theorem~2]{PS20}.\\

We record the following consequence of Lemma \ref{le3} that is interesting on its own.
\begin{corollary}\label{co1}
Let $K\subseteq [0,1]^d$ be non-empty and convex.
 Then, for any $\rho\in [0,1]$ we have
\[
\vol\big(\{\bsx\in\mathbb{R}^d\ \colon \ \dist(\bsx,\partial K)\le \rho\}\big)\le d\, 2^{d+4}\rho.
\]
\end{corollary}

\begin{remark}
The dependence of the upper bound on the dimension $d$ can be improved to be sub-exponential in $d$. Essentially, it is of order of magnitude $d^{c d^{2/3}}$ for some constant $c>0$. This can be achieved because the above upper estimate of the sum $$\sum_{j=1}^d\binom{d}{j}\kappa_j \quad \mbox{with the volumes of the unit balls $\kappa_j=  \frac{\pi^{j/2}}{\Gamma(1+j/2)}$}$$ can be replaced with the following bound (the lower bound is only presented as reference value):  
\[ \frac{1}{\sqrt{2\pi}{\rm e}d}  d^{\delta \, d^{2/3 - \delta}} \le \sum_{j=1}^d {d \choose j} \frac{\pi^{j/2}}{\Gamma(1+j/2)} \lesssim \left(d \sqrt{2 {\rm e}^3 \pi}\right)^{\kappa \,  d^{2/3}},\]
where $\delta \in (0,2/3)$ and $\kappa >{\rm e} (2\pi)^{1/3}$. The implied constant in the $\lesssim$ notation is absolute and can be figured out explicitly from the proof below.

\end{remark}

\begin{proof}
We use the estimates $$\frac{n^k}{k^k} \le {n \choose k} \le \frac{n^k}{k!} < \left(\frac{n {\rm e}}{k}\right)^k \quad \mbox{for $1 \le k \le n$,}$$ and Stirling's formula for the $\Gamma$-function which states that for all $x>0$ we have $$\Gamma(x)=\sqrt{\frac{2 \pi}{x}} \left(\frac{x}{{\rm e}}\right)^x {\rm e}^{\mu(x)},$$ where the function $\mu$ satisfies $0< \mu(x) < \frac{1}{12 x}$ for all $x>0$. 

We first show the upper bound. Using the above mentioned estimates we have
\begin{eqnarray*}
\sum_{j=1}^d {d \choose j} \frac{\pi^{j/2}}{\Gamma(1+j/2)} & \le & \frac{{\rm e}}{\sqrt{2 \pi}} \sum_{j=1}^d \frac{(d \sqrt{2 {\rm e}^3 \pi})^j}{j^{3j/2}}. 
\end{eqnarray*}
Let $\kappa > {\rm e}(2 \pi)^{1/3}$. Then we split the above sum into two parts and obtain this way
\begin{eqnarray*}
\sum_{j=1}^d {d \choose j} \frac{\pi^{j/2}}{\Gamma(1+j/2)} & \le & \frac{{\rm e}}{\sqrt{2 \pi}} \left[ \sum_{j=1 \atop j \le \kappa d^{2/3}}^d \frac{(d \sqrt{2 {\rm e}^3 \pi})^j}{j^{3j/2}}+  \sum_{j=1 \atop j > \kappa d^{2/3}}^d \frac{(d \sqrt{2 {\rm e}^3 \pi})^j}{j^{3j/2}}\right]\\
& \le &  \frac{{\rm e}}{\sqrt{2 \pi}} \left[ \left(d \sqrt{2 {\rm e}^3 \pi}\right)^{\kappa d^{2/3}} \sum_{j=1}^{\infty} \frac{1}{j^{3j/2}}+  \sum_{j=1 \atop j > \kappa d^{2/3}}^d \left(\frac{\sqrt{2 {\rm e}^3 \pi}}{\kappa^{3/2}}\right)^j\right]\\
& \lesssim & \left(d \sqrt{2 {\rm e}^3 \pi}\right)^{\kappa d^{2/3}}.
\end{eqnarray*}

Now we show the lower bound. Again using the above mentioned estimates gives
\[
\sum_{j=1}^d {d \choose j} \frac{\pi^{j/2}}{\Gamma(1+j/2)}  \ge \frac{1}{\sqrt{2\pi}{\rm e}}\sum_{j=1}^d \frac{d^j}{j^{3j/2}}.
\]
Put $x:=\frac{2}{3}-\delta$ with arbitrarily small $\delta \in (0,2/3)$. We estimate the above sum from below by considering only the summand $j=\lfloor d^x \rfloor$. This way we obtain
\begin{eqnarray*}
\sum_{j=1}^d {d \choose j} \frac{\pi^{j/2}}{\Gamma(1+j/2)} & \ge & \frac{1}{\sqrt{2\pi}{\rm e}}\frac{d^{\lfloor d^x\rfloor}}{\lfloor d^x\rfloor^{3 \lfloor d^x \rfloor /2}}\\
& \ge &  \frac{1}{\sqrt{2\pi}{\rm e}d}\frac{d^{ d^x}}{(d^x)^{3 d^x /2}}\\
& \ge & \frac{1}{\sqrt{2\pi}{\rm e}d} d^{\delta\,d^{(1-\delta)2/3}}.
\end{eqnarray*} 
\end{proof}

\section{Relation of the spectral test to the distance function of a lattice and the proof of  Theorem~\ref{thm2}}\label{sec4}

For the proof of Theorem \ref{thm2} we shall employ the recent characterization from  \cite[Theorem~1]{KS20}, see also Remark 8 there, which says that for any finite and nonempty point set $\cP\subseteq [0,1)^d$ the minimal error satisfies, for implied constants independent of $\cP$,

\begin{equation}\label{aserr}
e(\cP,d,s,p,q)\, \asymp_{d,s,p,q} \, \|\dist(\cdot,\cP)\|_{L_{\gamma}([0,1)^d)}^{s-d(1/p-1/q)_+},
\end{equation}
where $\gamma =s(1/q-1/p)^{-1}$ if $q<p$ and $\gamma=\infty$ if $q\ge p$. We remark that
\[
\|\dist(\cdot,\cP)\|_{L_{\infty}([0,1)^d)}=\sup_{\bsy\in [0,1)^d} \min_{\bsx\in \cP}\|\bsx-\bsy\|_2
\]
is the covering radius of $\cP$, which determines up to constants the quantity $e(\cP,d,s,p,q)$ in the range $q\ge p$ but is not sufficient to cover the range $q<p$. 

The proof of Theorem~\ref{thm2} is a combination of \eqref{aserr} with the following proposition which puts the spectral test into relation with an integral over the distance function to a lattice point set. Such integrals are studied in the context of lattice quantizers, see, e.g., Conway and Sloane \cite{CS82} as well as Graf and Luschgy \cite[Chapter 8]{GL00}.

\begin{proposition}\label{pro:distspectral}
Let $d \in \NN$. Then there exist numbers $0 < c_d < C_d$, only depending on $d$, with the following property: For every lattice point set $\mathcal{P}(L)$ in $[0,1)^d$  we have
\[
\frac{c_d}{2^{1/\gamma}}\,\sigma(L)\le \bigl\|\dist\bigl(\cdot,\mathcal{P}(L)\bigr)\bigr\|_{L_{\gamma}([0,1)^d)}\le C_d\, \sigma(L),
\]  
for every $\gamma \in (0,\infty]$.
\end{proposition}
\begin{proof}
H\"older's inequality implies $\|\dist(\cdot,\mathcal{P}(L))\|_{L_{\gamma}([0,1)^d)}\le \|\dist(\cdot,\mathcal{P}(L))\|_{L_{\infty}([0,1)^d)},$ 
and therefore it is sufficient to prove the lower bound for $\gamma \in (0,\infty)$  and the upper for $\gamma=\infty$.

We start with the proof of the lower bound and take a hyperplane covering $\mathcal{H}$ of $L$ with distance $\sigma(L)$ (see Hellekalek~\cite[Section 5.4]{H98}). For $t\in (0,1/2)$ consider the set
\[
A_t:=\bigg\{\bsx\in [0,1)^d \ \colon \ \dist\bigg(\bsx,\bigcup_{H\in\mathcal{H}}H\bigg)\ge t \, \sigma(L)\,\bigg\}.
\]
As the family $\mathcal{H}$ covers the lattice $L$, we have $\mathcal{P}(L)\subseteq \bigcup_{H\in\mathcal{H}}H$, and thus for $\bsx\in A_t$ it holds that $\dist(\bsx,\mathcal{P}(L))\ge t\, \sigma(L)$. Taking powers and integrals on both sides yields
\[
\int_{[0,1)^d}\dist(\bsx,\mathcal{P}(L))^{\gamma}\,{\rm d}\bsx\ge  t^{\gamma} \, \sigma(L)^{\gamma} \, \vol(A_t).
\]

For establishing the lower bound it suffices to find $t_d>0$ such that $\vol(A_{t_d})\ge 1/2$. To this end, we show that the volume of $B_{t}:=[0,1)^d\backslash A_{t}$ satisfies $\vol(B_{t_d})\le 1/2$ for some $t_d>0$. We first decompose the set $B_t$ into the disjoint union $B_t=\bigcup_{H\in\mathcal{H}} S_t(H),$
where we let $S_t(H):=\{\bsx\in [0,1)^d \ \colon \ \dist(\bsx,H)< t\, \sigma(L)\}$. Consequently, its volume is $\vol(B_t)=\sum_{H\in\mathcal{H}} \vol(S_t(H))$. 

For any $t>0$, at most $\sqrt{d}/\sigma(L)+2$ of the sets $S_t(H), H\in\mathcal{H},$ are non-empty, and thus only finitely many terms of the sum are non-zero. This is because the cube $[0,1)^d$ has diameter $\sqrt{d}$ and can therefore be intersected by no more than $\sqrt{d}/\sigma(L)$ hyperplanes contained in $\mathcal{H}$. The volume of a set $S_t(H)$ is bounded by its width, which is at most $2t\,\sigma(L)$ times the quantity $\sup_H \vol_{d-1}(H\cap [0,1)^d)$, where the supremum is extended over all hyperplanes $H$. Since $[0,1)^d$ is bounded, this supremum is bounded by some constant only depending on the dimension, call it $v_d$. This implies, since $H\in\mathcal{H}$ can be arbitrary,
\[
\vol(B_t)\le (\sqrt{d}/\sigma(L)+2)\, 2t \,\sigma(L)\, v_d=(2\sqrt{d}+4\sigma(L))\,v_d\, t.
\]
Using that $\sigma(L)\le \sqrt{d}$ we can choose $t_d=(12\sqrt{d}v_d)^{-1}$ such that $\vol(B_{t_d})\le 1/2$. This completes the proof of the lower bound.

We turn to the proof of the upper bound for which we have to find $C_d>0$ such that $\|\dist(\cdot,\mathcal{P}(L))\|_{L_{\infty}([0,1)^d)}\le C_d\,\sigma(L)$. Let $$\rho:=\frac{1}{2} \|\dist(\cdot,\mathcal{P}(L))\|_{L_{\infty}([0,1)^d)}.$$
Then there exists a ball $B(\bsy,\rho):=\bsy+\rho B$ with $\bsy\in [0,1)^d$ which is empty of points from $\mathcal{P}(L)$. By \cite[Lemmas 2 and 3]{KS20} there exists a ball $B(\bsz,\rho')$ contained in $ B(\bsy,\rho)\cap [0,1)^d$ with $\rho'\ge u_d \rho$, where the quantity $u_d>0$ only depends on $d$. Let $P$ be the fundamental parallelotope with respect to a reduced basis of $L$ such that $\diam(P)\le d 2^{d-1}\sigma(L)$ and consider the translate $\bsx+P$, $\bsx\in L$, containing the center $\bsz$ of the ball. If $\diam(P)\le \rho'$, then we must have the inclusions $\bsx+P\subseteq B(\bsz,\rho')\subseteq [0,1)^d$ and thus $\bsx\in \mathcal{P}(L)\cap B(\bsz,\rho')$, a contradiction to $B(\bsy,\rho)\cap \cP(L)=\emptyset$. Therefore, we must have $\diam(P)>\rho'\ge u_d \rho$ and $$ \|\dist(\cdot,\mathcal{P}(L))\|_{L_{\infty}([0,1)^d)}\le C_d\, \sigma(L)$$ for the quantity $C_d=d2^{d}u_d^{-1}$.
\end{proof}

\begin{remark}
An inspection of the proof of Proposition \ref{pro:distspectral} shows that it, and also Theorem~\ref{thm2}, remains valid if we consider point sets arising from the intersection of $[0,1)^d$ with any lattice, not necessarily containing $\mathbb{Z}^d$, under the assumption that these point sets are not concentrated on a single hyperplane.
\end{remark}

\bibliographystyle{plain}

\begin{small}
\noindent\textbf{Authors' addresses:}\\
\noindent Mathias Sonnleitner, Institut f\"ur Analysis, Johannes Kepler Universit\"{a}t Linz, Altenbergerstr.~69, 4040 Linz, Austria.\\
\textbf{E-mail:} \texttt{mathias.sonnleitner@jku.at}\\

\noindent  Friedrich Pillichshammer, Institut f\"{u}r Finanzmathematik und Angewandte Zahlentheorie, Johannes Kepler Universit\"{a}t Linz, Altenbergerstr.~69, 4040 Linz, Austria.\\
\textbf{E-mail:} \texttt{friedrich.pillichshammer@jku.at}\\
\end{small}

\end{document}